\documentclass[11pt,twoside,a4paper,notitlepage]{article}

\usepackage{amsmath}

\usepackage{amssymb}
\usepackage{amsfonts}
\usepackage[a4paper,margin=2cm,vmargin={1cm,2cm},includeheadfoot]{geometry}
\usepackage[T1]{fontenc}
\usepackage{amsthm}
\usepackage{enumerate}
\usepackage[scaled=0.92]{helvet}
\usepackage{srcltx}
\usepackage{graphicx}
\usepackage{tikz}
\usepackage{booktabs}

\usepackage[labelfont=bf,margin=1cm,justification=justified,labelsep=period]{caption}

\theoremstyle{plain}
\newtheorem{theorem}{Theorem}%[section]
\newtheorem{proposition}[theorem]{Proposition}

\newtheorem{property}[theorem]{Property}

{\theoremstyle{definition}
\newtheorem{remark}[theorem]{Remark}}

{\theoremstyle{definition}
}

\newcommand{\R}{{\mathbb {R}}}
\newcommand{\N}{{\mathbb{N}}}

\newcommand{\ah}{A_{\mathcal{H}}}
\newcommand{\af}{A_{\mathcal{F}}}
\newcommand{\ac}{A_{\mathcal{C}}}

\newcommand{\bs}{\boldsymbol}   % SIMBOLS NEGRETA
\definecolor{gris}{gray}{0.5}

\begin{document}

\begin{center}
{\LARGE\bf
 On the norming constants for  normal maxima} \\ [0.5 cm]

{\sc\Large Armengol Gasull, Maria Jolis  and Frederic Utzet} \\ [0.5 cm]
\end{center}

\noindent{\sl\large Departament de Matem\`{a}tiques, Edifici C,  Universitat Aut\`{o}noma de Barcelona,
08193 Bellaterra (Barcelona) Spain.  E-mails:  gasull@mat.uab.cat, mjolis@mat.uab.cat,
 utzet@mat.uab.cat}

\begin{center}
\begin{minipage}{\linewidth}
\begin{small}

\noindent{\bf Abstract.} In a remarkable paper, Peter Hall [{\it On
the rate of convergence of normal extremes}, J. App. Prob, {\bf 16}
(1979) 433--439] proved that the supremum norm distance between the
distribution function of the normalized maximum of $n$ independent
standard normal random variables and the distribution function of
the Gumbel law is bounded by $3/\log n$. In the present paper we
prove that choosing a different set of norming constants that bound
can be reduced to $1/\log n$. As a consequence, using the asymptotic
expansion  of a
 Lambert $W$ type function,
 we propose new explicit constants for the maxima
of normal random variables.

\smallskip


\smallskip


\end{small}
\end{minipage}
\end{center}

\section{Introduction}

 Let  $X_1,\dots,X_n$ be i.i.d. standard normal random variables
 and denote by $M_n$ its maximum,
$$M_n=\max\{X_1,\dots,X_n\}.$$
The normal law is in the domain of attraction for maxima of the
Gumbel law, that is,
 there are sequences of real numbers $\{a_n,\ n\ge 1\}$ and $\{b_n,\ n\ge 1\}$
(the {\it norming --or normalizing-- constants}) with $a_n>0$ such
that
\begin{equation}
\label{conv0}
 \lim_n \frac{1}{a_n}\,(M_n-b_n)= G, \ \text{in
distribution},
\end{equation}
where $G$ is a Gumbel random variable, with  distribution function
\begin{equation}
\label{gumbel-dist}
\Lambda(x)=\exp\{-e^{-x}\},\ x\in \R.
\end{equation}
Denote by $\Phi(x)$ the distribution function of a standard normal
law and by $\ \phi(x)$ its density. The convergence (\ref{conv0}) is
equivalent that for every $x\in \R$,
\begin{equation}
\label{conv2} \lim_n \Phi^n (a_nx+b_n)=\Lambda(x).
\end{equation}
In a remarkable paper Peter Hall
 \cite{Hall79}  proved that taking $b^*_n$ such that

$$\dfrac{1}{\sqrt{2\pi}}\,\dfrac{1}{b_n^*}\,e^{-(b_n^*)^2/2}=\frac{1}{n}\quad\mbox{and}\quad
a^*_n=1/b^*_n,$$
it holds  that for $n\ge 2$,
\begin{equation}
\label{DesHall}
\frac{C'}{\log n}<\sup_{x\in \R}\vert \Phi^n (a_n^*x+b_n^*)-\Lambda(x)\vert < \frac{C}{\log n},
\end{equation}
with $C=3$, and that the rate of convergence cannot be improved by
choosing a different sequence of norming constants. In this way,
Hall gives  a precise quantification of the remark  by Fihser and
Tippet in the   the seminal paper \cite{FisherTippett28}: {\it From
the normal distribution the limiting distribution is approached with
extreme slowness}. Notice that if $2\le  n\le 20$, then  $3/\log
n>1$, so the upper bound in (\ref{DesHall}) gives no information. It
should also be remarked that Hall \cite{Hall79} points out  that his
constant $C$ in (\ref{DesHall}) can be decreased to $0.91$ when
$n\ge 10^6.$

In the present paper we prove that taking
\begin{equation}
\label{anth}
b_n=\Phi^{-1}\big(1-\frac 1n\big)\quad \mbox{and}\quad a^\circ_n=\frac{b_n}{1+b_n^2},
\end{equation}
we have the following theorem.

\begin{theorem}
\label{maintheorem}
 Given $n_0\ge5,$ for all $n\ge n_0$ it
holds that
\begin{equation*}
\sup_{x\in \R}\vert \Phi^n \big(a^\circ_n\, x+b_n\big)-\Lambda(x)\vert
<\frac{C(n_0)}{\log n},\end{equation*} with
\begin{equation*}
C(n_0)=\begin{cases}1,&\quad\mbox{when}\quad n_0\le 15\\
 \Big(\dfrac{2}{3b_{n_0}^2} +\dfrac{1}{\sqrt{e}n_0}\Big)\log (n_0)<1&\quad\mbox{when}\quad n_0\ge
16.\end{cases}
\end{equation*}
Moreover\, $\lim_{n_0\to\infty} C(n_0)=1/3.$
\end{theorem}

The above result is quite sharp because our numerical analysis shows
that when $n_0$ moves in the range $[10^{20},10^{60}]$, then
$C(n_0)$ cannot be taken smaller than $0.12$, see
Table~\ref{comparison-t}. In Proposition~\ref{p4} we give some
bounds for $\{b_n^2\}$ that in particular prove that when
$n_0\ge16,$
\[
C(n_0)\le  \widetilde{C}(n_0)=\dfrac 1 3\dfrac{1}{1-\dfrac{\log(4\pi
\log n_0) )}{2\log n_0}} +\dfrac{\log n_0}{\sqrt{e}n_0},
\]
obtaining  explicit and simple computable upper bounds for $C(n_0)$.
To have an idea of how $C(n_0)$ and $\widetilde{C}(n_0)$ change with $n_0$ we
present some values in Table~\ref{t1}.

\begin{table}[htb]
\centering
\begin{tabular}{ccccccccccc}
\toprule
$n_0$&  $16$ & 30& 50& $10^2$ & $10^4$ & $10^6$ & $10^{10}$ & $10^{20}$&$10^{100}$\\
\midrule
$C(n_0)$  & 0.90&0.75 &0.67& 0.60 &  0.45&  0.41&   0.38 &0.36&0.34\\
\midrule
$\widetilde{C}(n_0)$  & 1.10&0.82 &0.72& 0.63 &  0.45&  0.41&   0.38 &0.36&0.34\\
\bottomrule
\end{tabular}
\caption{Several upper approximations for $C(n_0)$ and
$\widetilde{C}(n_0)$.}\label{t1}
\end{table}

From a practical point of view, in order to have explicit expressions of the
constants, it is suggested the following  asymptotic equivalents to
the norming constants $b^*_n$ and  $a_n^*$, respectively
 (Hall \cite[Diplay (4)]{Hall79}):
\begin{equation*}
 \beta^*_n=(2\log n)^{1/2}-\log(4\pi\log n)/\big(2(2\log n)^{1/2}\big)
\quad \text{and}\quad \alpha^*_n=1/\beta^*_n.
\end{equation*}
(It is also  proposed $\alpha^*_n=(2\log n)^{-1/2}$,
see, for example, Resnick \cite[pp. 71--72]{Res87}).
The expression of $\beta_n^*$ is easily deduced by observing that
$b^*_n$ can be expressed in terms of the Lambert W function (Corless
{\it et al.} \cite{CorGonHarJefKnu86}) and its well known
asymptotics, see  Section~\ref{se:5}.

However, in view of Theorem \ref{maintheorem} and our numerical
computations (see again Table~\ref{comparison-t}),  on the
one hand, it seems sensible
to approach accurately $b_n$, rather than $b^*_n$,  and  we propose
 the constant
\begin{equation}\label{betafinal}
\beta_n=\bigg(\log\big(n^2/(2\pi)\big)-\log\log\big(n^2/(2\pi)\big)+
\frac{\log\big(\log
(n^2)+1/2\big)-2}{\log\big(n^2/(2\pi)\big)}\bigg)^{1/2},
\end{equation}
that, as we will see, satisfies \[
b_n=\beta_n+O\left(\dfrac{(\log\log n)^2}{(\log
n)^{5/2}}\right),\quad n\to\infty.
\] On the other hand, by   Remark~\ref{comp} and
once more Table~\ref{comparison-t}, it seems convenient to use
$$\alpha_n=\frac{\beta_n}{1+\beta_n^2},$$ rather than $1/\beta_n$.
The expression of $\beta_n$ is derived from an asymptotic expansion
of $b_n$  using an approximation to Mills ratio by rational
functions (see Subsection \ref{sub:canonical}),   an extension of
the asymptotics of the Lambert function to a more general class of
functions (Subsection \ref{subsec:comtet}) and a final refinement
motivated by some numerical computations  (Subsection
\ref{subsec:numer}).

The paper is organized in the following way: In Section 2 there are recalled
some elementary facts about Extreme Value Theory for  normal random variables,
and there are  presented  graphical and numerical comparative studies of the performance of the constants
$a^*_n$ and $b^*_n$ versus  $a_n^\circ$ and $b_n$, and other proposals.
In Section 3 there are presented some technical preliminary results needed
in next sections. Section 4 is devoted to proof of Theorem \ref{maintheorem}.
Finally, in Section 5, new explicit expressions of the norming constant
are given.

\section{Extreme value theory for the normal law}
\label{sec:extremevalue}

By classical Extreme Value  Theory,   the norming constant $b_n$
 in (\ref{conv0})  can be taken, and we take, in agreement with the notation
(\ref{anth}),
\begin{equation}
\label{bn}
 b_n=\Phi^{-1}(1-n^{-1}),
\end{equation}
The constant $a_n$ can be chosen to be
\begin{equation}
\label{cn} a_n=A(b_n),
\end{equation}
where $A$ is an auxiliary function corresponding to $\Phi$ (see,
for example,  Resnick \cite[Proposition 1.11]{Res87}). Auxiliary
functions are not unique though they are asymptotically  equal.
Moreover, under certain conditions, an  auxiliary
function  is (see again  Resnick \cite[Proposition 1.11]{Res87}) the
 quotient of the survival function (one minus the distribution
function) and the density function, that is,
\begin{equation}
\label{aux-canon} \ac(x)=\frac{1-\Phi(x)}{\phi(x)},
\end{equation}
which is called the Mills ratio.  Since this function
is expressed in terms of the distribution function and the density
function, and does not depend on any other computation,  we call it
the {\it canonical} auxiliary function. We should remark that from
the standard proof of the convergence (\ref{conv0}) it is not
deduced that the constants $b_n$ and $\ac(b_n)$  produce more
accurate results than other constants computed with other auxiliary
functions or other ways.

To find manageable expression of the constants it is   used a
property of the convergence in law adapted to this context:

\begin{property}
\label{propietat1}
 With the preceding notations, if the sequences
$\{a_n',\, n \ge 1\}$ and $\{b_n',\, n \ge 1\}$ satisfy
$$
\lim _n \dfrac{a_n}{a'_n}=1\quad \text{and}\quad
 \lim_n\dfrac{b_n-b'_n}{a_n}=0,$$
 then
$$\lim_{n}\frac{1}{a'_n}\big(M_n-b'_n\big)=G\ \text{in distribution.}$$
\end{property}

Moreover, it is very  useful the following property  that involves the  use of  the norming constants of a simpler distribution
 function right tail equivalent to $\Phi$:

\begin{property} Let $F$ be a distribution function right tail equivalent
 to $\Phi$, that means,
$$\lim_{x\to \infty}\frac{1-\Phi(x)}{1-F(x)}=1.$$
Then the norming constants  of  $F$ and $\Phi$ can be taken equal.
\end{property}

\bigskip

 Thanks to the well known  asymptotics of the  Mills ratio,
$$\lim_{x\to\infty}\frac{1-\Phi(x)}{\dfrac{1}{\sqrt{2\pi}}\,\dfrac{1}{x}\,e^{-x^2/2}}=1,$$
 we can consider a distribution function $F$ such that there is  some
$x_0$, such that for $x>x_0,$
\begin{equation}
\label{F}
F(x)=1-\dfrac{1}{\sqrt{2\pi}}\,\dfrac{1}{x}\,e^{-x^2/2},
\end{equation}
and we deduce other possible constants: $b_n^*$ is given by
$$b_n^*=F^{-1}(1-n^{-1}),$$
or, equivalently, $b_n^*$ verifies
\begin{equation}
\label{bnHall}
\dfrac{1}{\sqrt{2\pi}}\,\dfrac{1}{b_n^*}\,e^{-(b_n^*)^2/2}=\frac{1}{n}.
\end{equation}
On the other hand, the canonical auxiliary function associated to $F$ is
\begin{equation}
\label{auxFT} \af(x)=\frac{x}{1+x^2}.
\end{equation}
We call this auxiliary function $\af$ because this  was the election of Fisher and
Tippett   \cite{FisherTippett28}. Note that $\ac$ and $\af$ are  asymptotically equivalent.
In our early  notations (\ref{anth}), we take
$$a_n^\circ=\af(b_n).$$

Furthermore,  it is typical to use a simpler  function
asymptotically  equivalent to both  $\af$ and $\ac$ given by
$$\ah(x)=\frac{1}{x}.$$
We write $$a_n^*:=\ah(b_n^*),$$
and we call $a^*_n$ and $b_n^*$  the {\it Hall's constants}.
 Hall didn't introduce such constants, that are classical (indeed, $b_n^*$ was proposed
 by Fisher and
Tippett \cite{FisherTippett28}), but as we commented,  Hall
 \cite{Hall79}  proved the  rate of convergence
(\ref{DesHall})  with these constants.
However, numerical studies
 show that
other norming constants give more accurate results that Hall's ones.
 In Figure
\ref{densitat100}
there is a plot of the Gumbel density and the density of
the random variables
$$Y^*_n=\frac{1}{a^*_n}(M_n-b^*_n)\quad \text{ and } \quad  Y_n=\frac{1}{a_n^\circ}(M_n-b_n).$$
%Y_n^{\cal C}=\frac{1}{\ac(b_n)}(M_n-b_n).$$
for $n=100$.

\begin{figure}[htb]
\centering
\includegraphics[width=8cm]{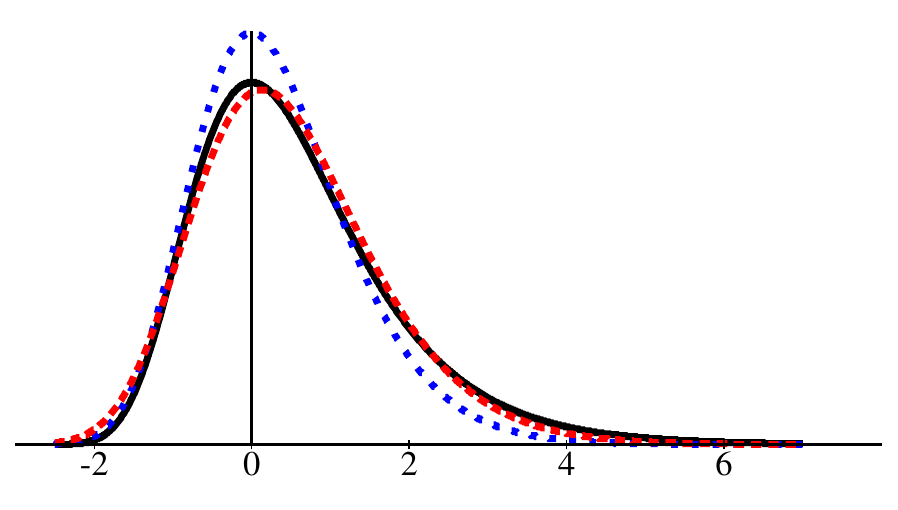}
\caption{Gumbel density and   density of the m\`{a}ximum of 100 standard Gaussian  random variables with different
norming constants. Solid line:
Gumbel density. Dotted blue line:  Density of  $Y^*_n$.
Dashed red line:  Density of $Y_n$.}
\label{densitat100}
\end{figure}

In order to assess the  velocity of convergence,  we numerically
compute an approximation of the distances $$\sup_{x\in \R}\vert
\Phi^n (A(b_n)\, x+b_n)-\Lambda(x)\vert\log (n),$$ for several
values of  $n$, and for different auxiliary functions $A$, and also
with Hall's constants $a^*_n$ and $b^*_n$, and  with the constants
proposed by Fisher and Tippet \cite{FisherTippett28} that are
$b_n^*$ and $A_{\cal F}(b^*_n)$. Those approximations are obtained
computing numerically the maxima of the corresponding functions. We
have used Maple.  The results are given in Table \ref{comparison-t}.

\begin{table}[htb]
\centering
\begin{tabular}{ccccccccc}
\toprule
& & &\multicolumn{6}{c}{$n$}\\
{Used in}&\multicolumn{2}{l}{\quad Constants}&  $10$ & $10^3$ & $10^{10}$ & $10^{30}$ & $10^{50}$ & $10^{60}$\\
\midrule
Theorem \ref{maintheorem} &$b_n$   & $a_n^\circ=A_{\cal F}(b_n)$   & 0.0420 & 0.1049 &  0.1191&  0.1208&   0.1208 &0.1207\\
Proposition \ref{posbis}&$b_n$   & $A_{\cal H}(b_n)$   & 0.3117 & 0.2752 & 0.2552 & 0.2465 & 0.2443 &0.2437 \\
&$b_n$   & $A_{\cal C}(b_n)$   & 0.1201 & 0.1260  & 0.1245 & 0.1224 & 0.1217 & 0.1215  \\
Fisher \& Tippett& $b^*_n$ & $A_{\cal F}(b_n^*)$ & 0.2331 & 0.2268 & 0.1997 & 0.1945 & 0.1938 & 0.1936  \\
Hall &$b^*_n$ & $a_n^*=A_{\cal H}(b^*_n)$ & 0.3546  &  0.3650&  0.3461 &  0.3354&  0.3324 & 0.3316 \\
\bottomrule
\end{tabular}
\caption{Approximate distance $\sup_{x\in \R} \vert \Phi^n (A(b_n)\,
x+b_n)-\Lambda(x)\vert\log(n)$ for different sample size $n$ and
different sets of norming constants.} \label{comparison-t}
\end{table}

%
%\begin{table}[htb]
%\centering
%\begin{tabular}{lcccccccc}
%\toprule
%& & &\multicolumn{6}{c}{$n$}\\
%{Users}&\multicolumn{2}{c}{Constants}&  $10$ & $10^2$ & $10^{10}$ & $10^4$ & $10^5$ & $10^6$\\
%\midrule
%&$b_n$   & $A_{\cal F}(b_n)$   & 0.0278 & 0.0197 &  0.0152&  0.0120&   0.0099 &0.0084\\
%&$b_n$   & $A_{\cal C}(b_n)$   & 0.0522 & 0.0272 & 0.0182 & 0.0137 & 0.0109 &0.0091 \\
%&$b_n$   & $A_{\cal H}(b_n)$   & 0.1354 & 0.0620  & 0.0398 & 0.0292 & 0.0230 & 0.0190  \\
%Hall &$b^*_n$ & $a_n^*=A_{\cal H}(b^*_n)$ & 0.1540  &  0.0801&  0.0528 &  0.0391&  0.0310 & 0.0256 \\
%Fisher \& Tippett& $b^*_n$ & $A_{\cal F}(b_n^*)$ & 0.1012 & 0.0512 & 0.0328 & 0.0237 & 0.0185 & 0.0151  \\
%\midrule
%& & $1/\log n$  & 0.4343 & 0.2171 & 0.1448 & 0.1086 & 0.0869 & 0.0724 \\
%\bottomrule
%\end{tabular}
%\caption{Approximate distance $\max_{x\in {\cal D}}
%\vert \Phi^n (A(b_n)\,
%x+b_n)-\Lambda(x)\vert$ for different sample size   $n$ and different sets
%or norming constants,
%   where ${\cal D}=\{-2.5,\, -2.49,\dots, 7\}$. In the last row there
%is a numeric value of $1/\log n$.}
%\label{comparison-t}
%\end{table}

Our purpose is to get theoretical explanations of  those numerical
results and  our main result is Theorem \ref{maintheorem} given in
the Introduction. We restrict  our study to  the cases where
$a_n=a_n^\circ=A_{\cal F}(b_n)$ and $a_n=A_{\cal H}(b_n)$, and we omit the
case $a_n=A_{\cal C}(b_n)$: the reasons for that omission  are the following:
First,  numerically and analytically $A_{\cal F}(b_n)$ is much
simpler that the $A_{\cal C}(b_n)$; second, Table \ref{comparison-t}
suggests that the performances of both $a_n=A_{\cal F}(b_n)$ and
$a_n=A_{\cal C}(b_n)$ are very similar; and finally, the study of
$A_{\cal C}(b_n)$ has its own details and tricks, and its study
would enlarge significantly  the paper.

%\begin{remark}
%\label{remark:theorem}
%By using the tools introduced to prove the above result we also will
%obtain that if we consider $\ac$ or $\ah$ rather than $\af$, then in
%the right hand side of the above inequality the constant should be
%taken at least 1.84, see Remark \ref{rem}. This  result illustrates
%that the choice $a_n=\af(b_n)$ seems to provide better norming
%constant for $a_n$ among $\af(b_n),\ac(b_n)$ and $\ah(b_n).$
%\end{remark}

\section{Preliminary results}

This section is divided in two parts. In the first one we  prove a couple
of
properties of the sequence $\{b_n\}$ which are needed in the sequel.  In the  second part we  introduce
the reciprocal of the canonical auxiliary function  that allows to
express in a  convenient way the difference $\Phi^n (a_n
x+b_n)-\Lambda(x)$.

\subsection{Bounds for $b_n$}

We prove  that the bounds for $(b_n^*)^2$ given  by Hall
\cite[display (2)]{Hall79}:
\begin{equation}
\label{fitesHall} 2\log n-\log(4\pi \log n) < (b_n^*)^2 <2\log n,
\end{equation}
are also satisfied for $b_n^2$.

\begin{proposition}\label{p4}
For each $n\ge 2$ the following inequalities hold:
\begin{equation}\label{cr-bn}
2\log n-\log(4\pi \log n)<b^2_n<2\log n.
\end{equation}

\end{proposition}
\begin{proof}
First of all, observe that for $n=2$ we have that $b_2=0$, while
$2\log 2-\log(4\pi \log 2)<0$ and $2\log 2>0$. So, we consider the
case $n\ge 3$. To see the right hand side inequality in (\ref{cr-bn}),
 we will prove that for $n\ge 3$,
$$1-\frac1n<\Phi\big(\sqrt{2\log n}\,\big).$$
By the change  of variables $y=\sqrt{2\log n}$, this inequality is
equivalent to
$$1-e^{-y^2/2}<\Phi(y),$$
for $ y\ge \sqrt{2\log 3}\approx 1.14823$. This is the same that
$$\int_{y}^{\infty}\frac1{\sqrt{2\pi}}\, e^{-x^2/2} dx<\int_{y}^{\infty}x\, e^{-x^2/2}
dx,$$ for $ y\ge \sqrt{2\log 3}$. And this inequality is clear
because $\frac1{\sqrt{2\pi}}\approx 0.3989$.

In order to prove  the inequality on the left-hand side of (\ref{cr-bn}),
 the
argument is similar. First, the function $ H(z):=2\log z-\log\big
(4\pi\log z)\big),\ z>1$, is strictly increasing, being negative for
$z=3$ and $z=4$, and positive for $z=5$. So we will prove the
inequality for
 $n\ge 5$.
 Define
$$h(z):=\sqrt{2\log z-\log(4\pi\log z)}=\sqrt{H(z)}.$$
It is clear that $h$ is strictly increasing and maps each interval
$[n,\infty)$ into $[h(n),\infty)$, for all $n\ge 5$. We must show
that  $\Phi(h(n))<1-1/n$ or, equivalently, that
$$\int_{h(n)}^{\infty}\dfrac1{\sqrt{2\pi}}\, e^{-x^2/2}\, dx>\frac1n=\int_n^{\infty}\frac1{y^2}\,dy.$$
By the change of variables $y=h^{-1}(x)$, the left hand side
of the above inequality is equal to
\begin{align*}
\int_n^{\infty}\dfrac1{\sqrt{2\pi}}\, & \exp\{-(2\log
y-\log(4\pi\log y)/2\}\, h'(y)\,dy
\\
&=\int_n^{\infty}\frac1{\sqrt{2\pi}}\, \frac1y\,\sqrt{4\pi\log y}\,
\frac1{2\sqrt{2\log y- \log(4\pi\log y)}}\Big[\frac2y-\frac{1}{y\log
y}\Big]\,dy
\\
&=\int_n^{\infty}\frac1{\sqrt{2}}\,\dfrac1{y^2}\,\frac{2\log
y-1}{\sqrt{\big(2\log y-\log(4\pi\log y)\big)\log y} }\,dy.
\end{align*}
To prove that this last term is greater than
$\int_n^{\infty}\frac1{y^2}\,dy$ we should prove that for any $y\ge
5$, $$ \frac{2\log y-1}{\sqrt{\big(2\log y-\log(4\pi\log y)\big)\log
y}}>\sqrt{2}.$$ By  the change of variables $u=\log y$ and
after squaring the two terms of the inequality and some
simplifications, we have to show that, for $u\ge \log 5$,
$$2u\log(4\pi u)>4u-1$$
that is the same that
$$g(u):=\log(4\pi u)-2+\frac1{2u}>0,$$
for $u\ge \log 5$. And this  is due to the fact that $g(\log 5)>0$,
and  $ g'(u)>0$ for $ u>\log 5$.

\end{proof}

In several parts of this work we will get the rate of convergence of
$\Phi^n(a_nx+b_n)$ to $\Lambda(x)$ in terms of $b_n^2$, and later we
translate it in terms of $\log n$. To this end, we  use the
following result:

\begin{proposition}\label{p5}
For any $n_0\ge 3$ and any $n> n_0$ the following inequality is
satisfied:
\begin{equation*}
b_n^2> K(n_0)\, \log n,\quad\mbox{with}\quad
K(n_0)=\dfrac{b_{n_{0}}^2}{\log n_0}.
\end{equation*}
%In particular, $n_0(0.6)=8,$ $n_0(0.84)=16$  and $n_0(0.9)=20.$  %$n_0(1)=32.$
\end{proposition}

\begin{proof}
Observe that the proposition is equivalent to say that the sequence
 $\{{b_n^2}/{\log n}\, ,n\ge 3\}$ is  increasing. Nevertheless,
we will prove the result in an indirect way. Specifically, we will
prove the following assertion:

\noindent { \it For any $K\in (0,2)$, the equation
\begin{equation}\label{eq1p5}
\Phi^{-1}(1-\frac1x)-\sqrt{K\log x}=0
\end{equation}
has a unique solution $ x_0>1$, and $
\Phi^{-1}(1-\frac1x)-\sqrt{K\log x}>0$
  for any $x>x_0$.}

Observe that if we take  $K=K(n_0)={b_{n_{0}}^2}/{\log n_0}$ (due to
Proposition \ref{p4},  the inequality  $0<K<2$ is satisfied), the
solution of the equation (\ref{eq1p5}) is precisely $x_0=n_0$, and
so, the proposition will follow from the assertion.

Composing with $\Phi$ and doing the change of variables
$y=\sqrt{K\log x}$, in order to prove the assertion, we must see
that the equation

\[
1-e^{-\frac{y^2}K}=\Phi(y),\] has a unique solution $ y_0>0$ and the
function $1-e^{-\frac{y^2}K}-\Phi(y)$ is positive for $y> y_0$.

Define
$$ M(y):=e^{-y^2/K}=\frac {2}K\int_y^\infty  x e^{-\frac{x^2}{K}}\,dx$$
and
$$N(y):=1-\Phi(y)=\frac {1}{\sqrt{2\pi}}\int_y^\infty
e^{-\frac{x^2}2}\,dx.$$ We have to prove that the equation
$M(y)-N(y)=0$ has a unique solution $ y_0>0$ and that for $y> y_0,$
$M(y)-N(y)<0$.
 To prove this, we study the
function $M(y)-N(y)$ for $y\ge 0.$ Notice that
\begin{equation}\label{mn}
\operatorname{sign}\left(\frac{M'(y)}{N'(y)}-1\right)=-\operatorname{sign}
(M'(y)-N'(y)).
\end{equation}
Therefore we introduce the function \[
g(y):=\frac{M'(y)}{N'(y)}=\frac{2\sqrt{2\pi}}{K}\,y\,
e^{\frac{K-2}{2K}\,y^2},\] and study the equation $g(y)=1,$ for
$y\ge0.$  It is not difficult to show that it has exactly two
solutions,  $y_1$ and $y_2$, and that
 \[0<y_1<\widetilde
y:=\sqrt{\frac{2-K}{K}}<y_2,\] where $\widetilde y$ is the unique
positive solution of $g'(y)=0.$

 Moreover $g(y)-1$ is positive in
$(y_1,y_2)$ and negative on $[0,y_1)\cup(y_2,\infty)$. Using
\eqref{mn}, we get that $M-N$ is increasing in
$[0,y_1)\cup(y_2,\infty)$ and decreasing in $(y_1,y_2)$.

Notice also that $M(0)-N(0)=1-1/2>0$. Joining all the information we
get that $M-N$ has at most one zero, $y=y^*,$ in $(0,y_2]$ and, if
exits, it is in $(y_1,y_2]$. In fact, since
\[
M'(x)<N'(x)\quad\mbox{for}\quad x> y_2,
\]
integrating both sides from $y$ to infinity, we obtain that
$M(y)-N(y)<0$ for all $y> y_2$.

In short, $M-N$ has exactly one zero $y_0=y^*$ in $[0,\infty)$, this
zero belongs to the interval $(y_1,y_2]$ and moreover $M-N$ is
negative for $y>y_0.$ This fact finishes the proof of the assertion.
\end{proof}

\subsection{The canonical auxiliary function and its reciprocal}
\label{sub:canonical}
The canonical auxiliary function
$$\ac(t)=\frac{1-\Phi(t)}{\phi(t)},\quad  t>0,$$
is known as Mills ratio and  enjoys nice properties. In Baricz \cite{bar} or Gasull and
Utzet \cite{gu}
it is proved that it is completely monotone, that means, the
derivatives alternate their signs: $\ac(t)>0$, and for $n\ge 1$,
$$(-1)^n\ac^{(n)}(t)>0,\ \text{for $t>0$}.$$
In particular $\ac$ is strictly decreasing and strictly convex. It
is also known how to construct two sequences of rational functions
 $\{P_n(t)/Q_n(t), \, n \ge 0\}$ with nonnegative integer
coefficients and numerators and denominators with increasing
degrees, such that for all $t>0$,
\begin{equation*} \frac{Q_{n+1}(t)}{P_{n+1}(t)}<\ac(t)<\frac{Q_{n}(t)}{P_{n}(t)},
\end{equation*}
see again \cite{gu} or \cite{p}.  We will use
\begin{equation}\label{fitesA}
\frac{t}{t^2+1}<\ac(t)<\frac{t^2+2}{t^3+3t}<\frac{1}{t}.
\end{equation}
%and
%\begin{equation}\label{nnova}
%\frac{t}{t^2+1}<\ac(t)<\frac{t^2+2}{t^3+3t}.
%\end{equation}

Denote by $V(t)$ the reciprocal of the canonical auxiliary  function
\begin{equation}
\label{reciprocal}
V(t)=\frac{1}{\ac(t)}
\end{equation}
Since $\ac$ is strictly decreasing,  $V$ is strictly increasing.
%Moreover,  it is also know that $V$ is strictly convex. This result
%was implicitly conjectured by Birnbaum \cite{Bir50} in 1950 and
%demonstrated few years later independently by  Sampford \cite{Sam}
%and Shenton \cite{se}, see also \cite{gu}.
Moreover, the bounds for $\ac$ give bounds  for~$V$. In particular,
from \eqref{fitesA},  for $t>0$,
\begin{equation}
\label{fitesV}
 t< V(t)<t+\frac{1}{t}.
\end{equation}

The function $V(t)$ also provides a very useful way to express the
function $\Phi^n (A(b_n)\, x+b_n)-\Lambda(x)$, which  is a main
object in this paper.

\begin{proposition}\label{dife} Set $ b_n=\Phi^{-1}(1-1/n)$ and
let $\{ a'_n,\, n\ge 1\}$ be an arbitrary sequence of strictly
positive numbers.  For every $x\in \R$ it holds that
\begin{equation}\label{lamlam}
\Phi^n (a'_n x+b_n)-\Lambda(x)=e^{-n
S_n(x)}\Big(\Lambda\big(I_n(x)\big)-\Lambda(x)\Big)+\Lambda(x)\left(e^{-nS_n(x)}-1
\right),
\end{equation}
where
\begin{equation}\label{condis}
I_n(x)=\int_{b_n}^{a'_nx+b_n} V(t) \, dt,\quad
0<S_n(x)<\frac{C_n^2(x)}{2(1-C_n(x))}\quad\mbox{and}\quad
C_n(x)=\frac1n e^{-I_n(x)}.
\end{equation}
\end{proposition}

\begin{proof} Notice that for $y\in\R$,
\begin{equation*}
1-\Phi(y)= \exp\big(\log(1-\Phi(y))\big)=\exp\Big\{\int_{-\infty}^y
\frac{-\phi(t)}{1-\Phi(t)}\, dt\Big\}=\exp\Big\{-\int_{-\infty}^y
V(t)\, dt\Big\}.\end{equation*} Then
\begin{align}
\label{abanslemma1}
1&-\Phi(a'_nx+b_n)=\exp\Big\{-\int_{-\infty}^{b_n} V(t)\, dt\Big\}\exp\Big\{-\int_{b_n}^{a'_nx+b_n} V(t)\, dt\Big\} \notag\\
&=
\exp\left\{\left.\log\big(1-\Phi(t)\big)\right|_{-\infty}^{b_n}\right\}\exp\Big\{-\int_{b_n}^{a'_nx+b_n}
V(t)\, dt\Big\} = \frac{1}{n} \,e^{-I_n(x)},
\end{align}
where the last equality follows from the definition of  $b_n$.
Notice also that, by \eqref{abanslemma1},
$0<\exp\big(-I_n(x)\big)/n<1.$

The following  formula is well-known and  was already used by Hall
in \cite{Hall79}. For $u\in (-1,1)$,
\begin{equation*}
\label{des-hall} \log(1-u)=-u-r(u)\quad\mbox{with}\quad 0\le r(u)\le
\frac{u^2}{2(1-u)}.
\end{equation*}
Then, from (\ref{abanslemma1}),
 \begin{equation}\label{ss}
\log\,\Phi^n(a'_nx+b_n)=n\, \log\Big(1-\frac{1}{n} \,
e^{-I_n(x)}\Big) =-e^{-I_n(x)}-nS_n(x),
 \end{equation}
 where $S_n(x)$ satisfies the
 conditions given in~\eqref{condis}.
 Hence,
\[\Phi^n(a'_nx+b_n)-\Lambda(x)=
 e^{-n
S_n(x)}\Lambda\big(I_n(x)\big)-\Lambda(x),
\]
and~\eqref{lamlam}  follows adding and subtracting the  term $ e^{-n
S_n(x)}\Lambda(x)$.
\end{proof}

\section{Proof of Theorem \ref{maintheorem}}
To short the notations in this proof, we  write $a_n$ instead of
$a_n^\circ=\af(b_n)$. We will consider separately the cases $x\ge0$
and $x<0,$ where the rate of convergence is analyzed; later both
cases are joined to get a global rate of convergence. In
Subsection~\ref{sub_other_an} there are some comments about the
other norming constants $\ah(b_n)$ and $\ac(b_n)$.
\subsection{ Case $x\ge0.$}

We  prove the following proposition:

\begin{proposition}\label{pos} Given $n_0\ge3,$ for all $n\ge n_0$ it
holds that
\begin{equation*}
\sup_{x\ge 0}\vert \Phi^n \big(a_n\,
x+b_n\big)-\Lambda(x)\vert <\frac{C^+(n_0)}{\log n},\end{equation*}
where \begin{equation*}
C^+(n_0)=\Big(\frac{1}{e\,b_{n_0}^2}+\frac{1}{2(n_0-1)}\Big)\log
(n_0).
\end{equation*}
\end{proposition}
\begin{proof}  By Proposition~\ref{dife}, for $x\ge0$, we have that
\begin{align}
\label{des:propo}
\big\vert\Phi^n (a_n x+b_n)-\Lambda(x)\big\vert&
\le e^{-n
S_n(x)}\Big\vert\Lambda\big(I_n(x)\big)-\Lambda(x)\Big\vert+\Lambda(x)\big\vert
e^{-nS_n(x)}-1
\big\vert\notag\\
&\le\Big\vert\Lambda\big(I_n(x)\big)-\Lambda(x)\Big\vert+\big\vert
e^{-nS_n(x)}-1 \big\vert.
\end{align}
We first study the term $\big\vert e^{-nS_n(x)}-1\big\vert$. From
(\ref{condis}), $0<C_n(x)\le{1}/{n}$ (now $I_n(x)\ge0$) and hence,
$$0<S_n(x)<\frac{1}{2n(n-1)}.$$
Thus, since when $y\ge0$, $ 1-e^{-y}\le y$, we get
\begin{align}\label{part1}
\big\vert
e^{-nS_n(x)}-1\big\vert&=1-e^{-nS_n(x)}<nS_n(x)<\frac{1}{2(n-1)}\nonumber\\&=\frac{\log
n}{2(n-1)}\frac 1{\log n}\le \frac{\log n_0}{2(n_0-1)}\frac 1{\log
n},\end{align} because the function $\log y/(y-1)$ is decreasing.
Notice that the above inequality gives the second term of the right
hand side of the statement.

To bound the other term we will study separately the cases whether
$I_n(x)<x$ or $I_n(x)\ge x.$ It can be seen that both situations
occur.

\smallskip

\noindent{\bf 1.} Case $I_n(x)<x$. Here,
\begin{align}\label{nnn}
\Big\vert\Lambda\big(I_n(x)\big)-\Lambda(x)\Big\vert&=
\Lambda(x)-\Lambda\big(I_n(x)\big)\le\Lambda'\big(I_n(x)\big)\big(x-I_n(x)\big)\nonumber\\
&=\Lambda\big(I_n(x)\big)e^{-I_n(x)}\big(x-I_n(x)\big)\le
e^{x-I_n(x)}e^{-x}\big(x-I_n(x)\big).
\end{align}
where we have used that for $x>0,$ $\Lambda(x)$ is increasing,
$\Lambda'(x)$ is decreasing,  the Mean Value Theorem and that
$\Lambda\big(I_n(x)\big)\le1$.

At this point observe that since $a_n=b_n/(b_n^2+1)$,
\[
0<x-I_n(x)\le x-\int_{b_n}^{a_nx+b_n} t \,
dt=x-\frac{(a_nx)^2}2-a_nb_nx\le (1-a_nb_n)x=\frac x{b_n^2+1},
\]
where we utilize the bound $V(t)> t$ given in \eqref{fitesV}. Hence,
plugging the above inequality in~\eqref{nnn},
\begin{align}\label{part2}
\Big\vert\Lambda\left(I_n(x)\right)-\Lambda(x)\Big\vert&\le e^{\frac
x{b_n^2+1}}e^{-x}\frac x{b_n^2+1} = e^{-\frac {b_n^2
x}{b_n^2+1}}\frac x{b_n^2+1}\nonumber\\
& = e^{-\frac {b_n^2 x}{b_n^2+1}}\frac
{b_n^2x}{b_n^2+1}\frac1{b_n^2}\le \max_{y\in[0,\infty)}\left\{
e^{-y}y\right\}\frac 1{b_n^2}=\frac1{eb_n^2}.
\end{align}

\smallskip

\noindent{{\bf 2.} Case $I_n(x)\ge x$.} Here,
\begin{align}\label{nnn2}
\Big\vert\Lambda\big(I_n(x)\big)-\Lambda(x)\Big\vert&=
\Lambda\big(I_n(x)\big)-\Lambda(x)\le\Lambda'(x)\big(I_n(x)-x\big)\nonumber\\
&=\Lambda(x)e^{-x}\big(I_n(x)-x\big)\le e^{-x}\big(I_n(x)-x\big) ,
\end{align}
where we have used the same properties as above. We proceed also
as in the previous situation, and moreover we use that
 when $y\ge0$, then  $\log(1+y)\le y$. Given that $V(t)<t+1/t$ (see \eqref{fitesV}),
\begin{align}\label{bb}
 0&\le I_n(x)-x\le
\frac{1}{2}\,a_n^2x^2+a_nb_n x+
\log\Big(\frac{a_n}{b_n}\,x+1\Big)-x\nonumber\\
&\le \frac{b_n^2}{2(b_n^2+1)^2}\, {x^2}+ \frac{b_n^2}{b_n^2+1}\,
x+\frac{1}{b_n^2+1}\, x-x= \frac{b_n^2}{2(b_n^2+1)^2}\,
{x^2}<\frac{1}{2 b_n^2}\, x^2.
\end{align}
Hence, from~\eqref{nnn2},
\begin{equation}\label{part3}
\Big\vert\Lambda\big(I_n(x)\big)-\Lambda(x)\Big\vert\le
\frac{x^2}{2}\,e^{-x}\,\frac{1}{b_n^2}<
 \frac{2}{e^2b_n^2}.
\end{equation}
The proposition follows joining \eqref{part1}, \eqref{part2} and
\eqref{part3}, using that $1/e>2/e^2$ and applying
Proposition~\ref{p5}.
\end{proof}

\bigskip

\subsection{ Case $x<0.$}

\begin{proposition}\label{neg} Given $n_0\ge3,$ for all $n\ge n_0$ it
holds that
\begin{equation*}
\sup_{x<0}\vert \Phi^n \big(a_n\, x+b_n\big)-\Lambda(x)\vert
<\frac{C^-(n_0)}{\log n},\end{equation*} where
\begin{equation*}
C^-(n_0)=\begin{cases}1,&\quad\mbox{when}\quad n_0\le 15\\
 \Big(\dfrac{2}{3b_{n_0}^2} +\dfrac{1}{\sqrt{e}n_0}\Big)\log (n_0)&\quad\mbox{when}\quad n_0\ge
16.\end{cases}
\end{equation*}
\end{proposition}
\begin{proof} First notice that for $x<0$, $\Lambda(x)<\Lambda(0)<1/e$, and
$$\Phi^n\big(a_nx+b_n\big)<\Phi^n\big(b_n\big)=\Big(1-\frac{1}{n}\Big)^n<\frac{1}{e}.$$
Hence, for\, $3\le n\le 15$,
\begin{equation}\label{f7}
\sup_{x<0}\Big\vert
\Lambda(x)-\Phi^n\big(a_nx+b_n\big)\Big\vert<\frac{1}{e}<\frac{1}{\log
n}.\end{equation} From now on we assume $n\ge 16$. For convenience,
we divide the values of $x$ according whether
\[
x\in(-\infty,-{b_n}/{a_n}),\quad x\in[-{b_n}/{a_n},-1.25\log b_n]
\quad\mbox{or}\quad x\in(-1.25\log b_n,0).
\]
 Notice that for
$n\ge16$, $-b_n/a_n<-1.25\log b_n.$ We remark that in our approach
the choice of the point $-1.25\log b_n$ is essential to obtain sharp
bounds.

\bigskip

\noindent{\bf 1.} Case $x\in(-\infty,-{b_n}/{a_n})$. Here
we will bound separately $\Lambda(x)$ and  $\Phi^n(a_nx+b_n)$.
Since $a_nx+b_n<0$, on the one hand,
$\Phi\big(a_nx+b_n\big)<\Phi(0)=0.5,$ and therefore
\begin{equation*}
0<\Phi^n(a_nx+b_n)<\frac 1 {2^n}.
\end{equation*}
On the other hand,
\begin{equation*}
0<\Lambda(x)<\Lambda\left(-\frac{a_n}{b_n}\right)
=\Lambda\left(-(b_n^2+1)\right)=\exp\left\{ -e^{b_n^2+1}\right\}.
\end{equation*}
Joining the above inequalities we obtain that for
$x\in(-\infty,-{b_n}/{a_n})$,
\begin{equation}\label{ff1}
\Big\vert \Lambda(x)-\Phi^n\big(a_nx+b_n\big)\Big\vert\le \max\Big(
\Lambda(x),\Phi^n\big(a_nx+b_n\big)\Big)<\frac 1{2^n}
\end{equation}
where we have used that for $n\ge3,$
\[
\exp\left\{ -e^{b_n^2+1}\right\}\le \frac 1{2^n}.
\]
It is easy to see that this  inequality holds for $n\le5.$ To prove
it for $n\ge 6$, notice that it is equivalent to see that
\[
 e^{b_n^2+1}\ge (\log 2) n.
\]
Now, by Proposition~\ref{p4},
\[
e^{b_n^2+1}\ge e^{1+2\log n-\log(4\pi\log n)}=\frac{e\,n^2}{4\pi\log
n}.
\]
Hence, it suffices to prove that,  for $n\ge6$, ${e\,n^2}/(4\pi\log
n)\ge (\log 2)n$ and this result follows studying the function
$y/\log y$ and its derivatives.

Finally, since  for $y>1$ the function $(1/2)^y\log y$ is
decreasing,  inequality~\eqref{ff1} implies that
\begin{equation}\label{ff2}
\Big\vert \Lambda(x)-\Phi^n\big(a_nx+b_n\big)\Big\vert<\frac{\log
n_0}{2^{n_0}}\frac 1{\log n}.
\end{equation}

\bigskip

\noindent{\bf 2.} Case $x\in[-{b_n}/{a_n},-1.25\log n]$. As in the
first case we will bound separately $\Lambda(x)$ and
$\Phi^n(a_nx+b_n)$.

 We start studying $\Lambda(x).$ We get that
\begin{equation}\label{f4}
0\le\Lambda(x)\le \Lambda(-1.25\log b_n)<\Lambda(-\log b_n)\le
\frac{4}{e^2}\,\frac1{b_n^2}\le \frac{4 \log n_0}{e^2 b_{n_0}^2}\,\frac
1{\log n}.
\end{equation}
where we have used Proposition~\ref{p5} and that
\begin{equation}\label{eee}
\Lambda(-y)\le \frac4{e^2}\,e^{-2y}
\end{equation}
for all $y\ge0$. Notice that this inequality holds because
$\max_{z\ge0} \big(z^2e^{-z}\big)=4/e^2.$

Let us consider now the  term $\Phi^n(a_nx+b_n)$. Recall  that
by~\eqref{ss},
 \begin{equation}
\log\,\Phi^n(a_nx+b_n)=-\exp\Big\{-\int_{b_n}^{a_nx+b_n} V(t) \,
dt\Big\}-nS_n(x),
 \end{equation}
with $S_n(x)>0.$ Hence, \begin{equation}\label{f1}
\Phi^n(a_nx+b_n)\le e^{-\exp \Big\{\int_{a_n x+b_n}^{b_n} V(t) \,
dt\Big\}  }.
\end{equation}
By \eqref{fitesV} we get
\begin{equation*}\label{eeeee}
\int_{a_nx+b_n}^{b_n} V(t)\,dt\ge\int_{a_nx+b_n}^{b_n} t\,dt=-a_nb_n
x-a_n^2\frac{x^2}2:=g(x).
\end{equation*}
Notice that on the interval $[-b_n/a_n,-1.25\log b_n]$ the
function $g$ is decreasing and then for all $x$ in this interval,
$g(x)\ge g(-1.25\log b_n).$ Thus,
\begin{align}\label{f2}
\int_{a_nx+b_n}^{b_n} V(t)\,dt&\ge \frac 54a_nb_n\log b_n
-\frac{25}{32}a_n^2(\log b_n)^2\nonumber\\
&= \frac54\frac{b_n^2}{b_n^2+1}\log b_n- \frac{25
b_n^2}{32(b_{n}^2+1)^2}(\log b_{n})^2\nonumber\\
&\ge  \frac54\frac{b_{n}^2}{b_{n}^2+1}\log b_n-
\max_{y>1}\left(\frac{25 y^2}{32(y^2+1)^2}\log y\right)\log
b_n\nonumber\\&\ge \frac54\frac{b_{n}^2}{b_{n}^2+1}\log b_n-
\frac1{10}\log b_n=\frac{23b_n^2-2}{20(b_n^2+1)}\log b_n,
\end{align}
where we have used that
\begin{equation}\label{f8}
\max_{y\ge1}\left(\frac{25 y^2}{32(y^2+1)^2}\log
y\right)<\frac1{10},
\end{equation}
and that $b_n>1$ for $n\ge16.$ The above inequality follows studying
the function $h(y)= (y^2\log y)/(y^{ 2}+1)^2.$ In fact,
\[
h'(y)=y\frac{y^2+1+2(1-y^2)\log y}{(y^2+1)^3},
\]
and its sign, when $y>1,$  is the contrary of the sign of
\[
H(y)=\log y-\frac{y^2+1}{2(y^2-1)},
\]
which can be easily studied because $H'(y)=(y^4+1)/(y(y^2-1)^2)>0.$
Hence, for $y>1,$ the function $h$ is positive and increasing until
some value $y=y^*$ and then decreases monotonically towards zero. By
Bolzano's Theorem,  $y^*\in (\underline y, \overline y):=
(2.16,2.17).$ Hence
\[
\max_{y>1} h(y)< \frac{\overline y\,^2\log \overline y}{(\underline
y^{ 2}+1)^2} \quad \mbox{and}\quad \frac{25}{32}\frac{\overline
y\,^2\log \overline y}{(\underline y^{ 2}+1)^2}<\frac1{10}.
\]
Combining \eqref{f1}, \eqref{f2} and \eqref{eee} we obtain that \[
\Phi^n(a_nx+b_n)\le \Lambda\left(-\int_{a_nx+b_n}^{b_n} V(t)\,dt
\right)\le \frac 4{e^2} \exp\left\{
\frac{2-23b_n^2}{10(b_n^2+1)}\log b_n  \right\}.
\]
Hence, once we prove that
\begin{equation}\label{falta1}
\max_{y\ge 1}  P(y) < \frac 23,
\end{equation}
where
\[
P(y)=\frac 4{e^2} y^2 \exp\left\{ \frac{2-23y^2}{10(y^2+1)}\log y
\right\}=\frac 4{e^2} y^2 Q(y),
\]
 we will have that
\begin{equation*} \Phi^n(a_nx+b_n)\le \frac
{2}{3b_n^2}\le
 \frac{2 \log n_0}{3b_{n_0}^2}\frac
1{\log n}.
\end{equation*}
where note that we have used once more Proposition~\ref{p5}.

Joining \eqref{f4} and the above inequality we get that when
$x\in[-{b_n}/{a_n},-1.25\log n]$,
\begin{equation}\label{f11}
 \big\vert \Phi^n(a_nx+b_n)-\Lambda(x) \big\vert \le
\frac{2 \log n_0}{3b_{n_0}^2}\frac1{\log n},
\end{equation}

Hence, to end this part of the proof we need to
prove~\eqref{falta1}.  To study the function $P(y)$ we compute
\[
P'(y)=\frac2{5e^2}\frac{y}{(1+y^2)^2}\,Q(y)\,q(y),\quad\mbox{with}\quad
q(y)=(22-3y^2)(1+y^2)-50y^2\log y.
\]
Moreover, for $y\ge1,$ $q'(y)=-100\, y\log y-12(1+y^2){ y}<0.$
Joining all this information we get that for $y\ge1,$ the function
$P'(y)$ is decreasing and has a unique zero $y^*$ and, by Bolzano's
Theorem, $y^*\in(\underline y,\overline y):=(1.532,1.533).$
Therefore the function $P$ is increasing in $[1,y^*)$ and decreasing
in $(y^*,\infty)$. As a consequence,
\[
\max_{y\ge 1} P(y)=P(y^*)<\frac 4{e^2} \overline y^2 \exp\left\{
\frac{2-23\underline y^2}{10(\overline y^2+1)}\log \underline
y\right\}<0.66<\frac23,
\]
as we wanted to prove.

\bigskip

\noindent{\bf 3.} Case $x\in(-1.25\log n,0)$.  Using
Proposition~\ref{dife} we write

\begin{equation*}
\Phi^n (a_n x+b_n)-\Lambda(x)=e^{-n
S_n(x)}\Lambda(x)\Big(\frac{\Lambda\left(I_n(x)\right)}{\Lambda(x)}-1\Big)+\Lambda(x)\left(e^{-nS_n(x)}-1
\right).
\end{equation*}
Hence
\begin{equation}\label{phi_lambda2}
\left\vert \Phi^n (a_n
x+b_n)-\Lambda(x)\right\vert\le\Lambda(x)\Big\vert\frac{\Lambda\left(I_n(x)\right)}{\Lambda(x)}-1\Big\vert
+\Lambda(x)\left\vert e^{-nS_n(x)}-1 \right\vert.
\end{equation}

 We start proving that
\begin{equation}\label{oblit}
-I_n(x)=\int_{a_nx+b_n}^{b_n} V(t)\,dt\le -x.
\end{equation}
Recall that by \eqref{fitesV}, for $t>0,$ $V(t)\le t+1/t$ and
moreover that $V$ is an increasing function. Hence
\[
\int_{a_nx+b_n}^{b_n} V(t)\,dt\le -V(b_n)a_nx\le
-\Big(b_n+\frac1{b_n}\Big)a_nx=-\frac{b_n^2+1}{b_n}a_nx=-x.
\]
Therefore, for the first term of the right hand side of
\eqref{phi_lambda2} we have
\begin{align}\label{lam5}
\Lambda(x)\Big\vert \frac{\Lambda\left(I_n(x)\right)}{\Lambda(x)}
-1\Big\vert&=\Lambda(x)\left\vert\exp\left(e^{-x}-e^{-I_n(x)}
 \right) -1\right\vert\nonumber\\
&=\Lambda(x)\left\{\exp\left(e^{-x}-e^{-I_n(x)}
 \right) -1\right\}\nonumber\\
 &\le \Lambda(x)\exp\left\{e^{-x}-e^{-I_n(x)} \right\}\left(e^{-x}-e^{-I_n(x)}
 \right),
 \end{align}
where in the last inequality we have applied the Mean Value Theorem
to $e^x$.

Now, notice  that taking into account that for $y\ge0,$ $1-e^{
-y}\le y,$
\begin{align}\label{uf}
e^{-x}-e^{-I_n(x)}&=e^{-x}\left(1- e^{x+\int_{a_nx+b_n}^{b_n}
V(t)\,dt}\right)\le e^{-x}\left(-x-\int_{a_nx+b_n}^{b_n}
V(t)\,dt\right)\nonumber\\&\le
e^{-x}\left((a_nb_n-1)x-\frac{a_n^2}2x^2\right)\le
e^{-x}\Big(-x+\frac{x^2}2\Big)\frac1{b_n^2},\end{align} where in the
last inequalities we have used first that $V(t)>t,$ and later that
$a_n^2$ and $1-a_nb_n$ are both smaller than $1/b_n^2.$

To continue, notice that since $-x\le 1.25\log b_n,$ then $b_n^2\ge
\exp(-8x/5).$ Hence, from~\eqref{lam5} and~\eqref{uf},
\[
\Lambda(x)\Big\vert \frac{\Lambda\left(I_n(x)\right)}{\Lambda(x)}
-1\Big\vert\le Q(x)\frac1{b_n^2},
\]
where
\[
Q(x)=\Big(-x+\frac{x^2}2\Big)\exp\left\{-x-e^{-x}+
\Big(-x+\frac{x^2}2\Big)e^{3x/5}\right\}=\Big(-x+\frac{x^2}2\Big)T(x).
\]
We claim
\begin{equation}\label{claim}
0<\max_{x<0} Q(x)<0.63.
\end{equation}
Therefore
\begin{equation}\label{des1}
\Lambda(x)\Big\vert \frac{\Lambda\left(I_n(x)\right)}{\Lambda(x)}
-1\Big\vert\le \frac{0.63}{b_n^2}.
\end{equation}

Let us prove now the inequality \eqref{claim} given in the above
claim. Notice that
\begin{equation*}
Q'(x)=\frac1{20}T(x)\,t(x)=\frac1{20}T(x)\,\Big(t_0(x)+t_1(x)e^{-x}+t_2(x)e^{3x/5}\Big),
\end{equation*}
where $t_0(x)=-10(x^2-4x+2)$, $t_1(x)=10x(x-2)$ and
$t_2(x)=x(x-2)(3x^2+4x-10).$  To study the sign of $Q'(x)$ we
consider the function $t(x)$ for $x<0.$  We get that
\[
t''(x)=-20+s_1(x)e^{-x}+s_2(x)e^{3x/5},
\]
where $s_1(x)=10(x^2-6x+6)>0$ and
$s_2(x)=(27x^4+342x^3+558x^2-1200x-300)/25.$ It is clear that
\[
t''(x)\ge -20+s_1(x)+s_2(x)e^{3x/5}=s_0(x)+s_2(x)e^{3x/5}=S(x),
\]
with $s_0(x)=10(x^2-6x+4).$ Hence, if we prove that $S(x)>0$ we will
have the convexity of $t(x)$. Joining this information with the fact
that $t(0)=-20<0, t'(0)=40>0$ and that $t(x)$ tends to infinity when
$x$ goes to minus infinity, we obtain that $t(x)$ has a unique zero
$x=x^*$ in $(-\infty,0).$ By Bolzano's Theorem $x^*\in(\underline
x,\overline x):=(-1.051,-1.050).$ Finally $Q$ is increasing on
$(-\infty,x^*)$ and decreasing in $(x^*,0)$ and therefore,
\[
\max_{x< 0} Q(x)=Q(x^*)<\Big(-\underline x+\frac{\underline
x^2}2\Big)\exp\left\{-\underline x-e^{-\overline x}+
\Big(-\underline x+\frac{\underline x^2}2\Big)e^{3\overline
x/5}\right\}<0.63,
\]
as we wanted to prove. That $S(x)>0$ for $x<0,$ can be proved by
using similar arguments and we omit the details.

To end the proof it remains to study the term $\Lambda(x)\vert
e^{-nS_n(x)}-1\vert,$ where recall that from Proposition~\ref{dife},
\[
S_n(x)\le \frac{C_n^2(x)}{2(1-C_n(x))}\quad \mbox{and}\quad
C_n(x)=\frac{1}{n} \,\exp\Big\{\int^{b_n}_{a_nx+b_n} V(t) \,
dt\Big\}=1-\Phi(a_nx+b_n).\] Hence
\[
nS_n(x)\le \frac{\left( \exp\Big\{\int^{b_n}_{a_nx+b_n} V(t) \,
dt\Big\}\right)^2  }{2n\Phi(a_nx+b_n)}\le \frac 1n  e^{-2x},
\]
where we have used~\eqref{oblit} and that $\Phi(a_nx+b_n)>1/2,$
because $a_nx+b_n>0.$

Using this inequality, that $1-e^{-y}\le y$ for $y\ge0$ and that
$e^{-x}\ge 1-x+x^2/2,$ for $x\le0$,
 we obtain
\begin{align}\label{des2}
\Lambda(x)\big\vert e^{-nS_n(x)}-1\big\vert&\le \frac 1n
\Lambda(x)e^{-2x}\le \frac 1n e^{-1+x-x^2/2} e^{-2x}\nonumber\\ &\le
\frac 1n e^{-1/2} e^{-(x+1)^2/2}\le\frac{\log
n}{\sqrt{e}n}\frac1{\log n}\le \frac{\log
n_0}{\sqrt{e}n_0}\frac1{\log n}.
\end{align}
Joining \eqref{des1} and \eqref{des2} and using Proposition~\ref{p5}
we arrive to
\begin{equation}\label{f10}
 \big\vert \Phi^n(a_nx+b_n)-\Lambda(x) \big\vert \le \left(
\frac{0.63\log n_0}{b_{n_0}^2} +\frac{\log
n_0}{\sqrt{e}n_0}\right)\frac1{\log n},
\end{equation}
for the values of $x$ considered in this case.

Finally, collecting the right hand terms of inequalities
\eqref{ff2}, \eqref{f11} and \eqref{f10} we have that for $n\ge16,$
\[
\max\left(\frac{\log n_0}{2^{n_0}},\frac{2\log
n_0}{3b_{n_0}^2},\Big( \frac{0.63\log n_0}{b_{n_0}^2} +\frac{\log
n_0}{\sqrt{e}n_0}\Big) \right)<\frac{2\log n_0}{3b_{n_0}^2}
+\frac{\log n_0}{\sqrt{e}n_0}.
\]
Hence
\[
 \big\vert \Phi^n(a_nx+b_n)-\Lambda(x) \big\vert < \Big( \frac{2\log n_0}{3b_{n_0}^2} +\frac{\log
n_0}{\sqrt{e}n_0}\Big)\frac 1{\log n},
\]
and the proposition follows.
\end{proof}

%\bigskip
\subsection{Global rate  of convergence: proof of Theorem \ref{maintheorem}}
%\noindent{\bf  Proof of Theorem~\ref{maintheorem}:}
 The first part
of the Theorem \ref{maintheorem}  is a straightforward consequence
of Propositions~\ref{pos} and~\ref{neg}: For $n_0\ge16$,  because it
is easy to prove that $C^-(n_0)>C^+(n_0)$. For $5\le n_0\le 15$
because $C^-(n_0)=1$ and $C^+(n_0)<1.$

Proposition \ref{p4} provides upper and lower bounds for $b_n^2$.
These bounds substituted in $C(n_0)$ give easily that
$\lim_{n_0\to\infty} C(n_0)=1/3.$ \qed

\subsection{Other norming constants $a_n$}
\label{sub_other_an}
If instead of $a_n^\circ=\af(b_n)=b_n/(1+b_n^2)$ it is used
$\ah(b_n)=1/b_n$, applying similar  tools that in the proof of
Proposition~\ref{pos} we get the following result:

\begin{proposition}\label{posbis} Given $n_0\ge2,$ for all $n\ge n_0$ it
holds that
\begin{equation*}
\sup_{x\ge 0}\vert \Phi^n \big(\ah(b_n)\,
x+b_n\big)-\Lambda(x)\vert <\frac{\overline{C}^+(n_0)}{\log
n},\end{equation*} with \begin{equation*} \overline{C}^+(n_0)=
\frac{\sqrt2+1}{e^{\sqrt 2}} \frac{\log n_0}{b_{n_0}^2}+\frac{\log
n_0}{2(n_0-1)}.
\end{equation*}
\end{proposition}
\begin{proof} Starting as in the proof of Proposition~\ref{pos} we
obtain that
\begin{equation}\label{zz}
\big\vert\Phi^n (\ah(b_n)
x+b_n)-\Lambda(x)\big\vert\le\Big\vert\Lambda\left(
{I}_n(x)\right)-\Lambda(x)\Big\vert +\frac{\log
n_0}{2(n_0-1)}\frac{1}{\log n},
\end{equation}
where here
\[
 {I}_n(x)=\int_{b_n}^{\ah(b_n)x+b_n}
V(t)\,dt,\quad\mbox{and}\quad \ah(b_n)=\frac 1{b_n}.\]

To study the remainder left hand term of~\eqref{zz} let us prove
first that under our hypotheses $ {I}_n(x)>x$. Notice that since
$\ac(t)<1/t$, $\ac(b_n)<\ah(b_n)$. Then, by the Mean Value Theorem,
there is $x_1\in [0,x]$ such that
\begin{align*}
 {I}_n(x)&=\int_{b_n}^{\ah(b_n)x+b_n} V(t) \,
dt>\int_{b_n}^{\ac(b_n)x+b_n} V(t) \, dt
\\&=V\big(\ac(b_n)x_1+b_n\big)\ac(b_n)\, x=
 \frac{\ac(b_n)}{\ac\big(\ac(b_n)x_1+b_n\big)}\, x>x,
\end{align*}
where in the last step we have used that $\ac$ is decreasing. Then
\begin{align}\label{nnn3}
\Big\vert\Lambda\left( {I}_n(x)\right)-\Lambda(x)\Big\vert&=
\Lambda\left( {I}_n(x)\right)-\Lambda(x)\le\Lambda'(x)
\big( {I}_n(x)-x\big)\nonumber\\
&=\Lambda(x)e^{-x}\big( {I}_n(x)-x\big)\le e^{-x}\big(
{I}_n(x)-x\big),
\end{align}
using once more  that for $x>0,$ $\Lambda(x)$ is increasing, and the
Mean Value Theorem.

Now, taking into account that $V(t)\le t+1/t$ (see (\ref{fitesV}))
and again that for $y>-1,$   $\log(1+y)\le y,$ we obtain that
\begin{equation}
\label{desigualtat_int}  {I}_n(x)-x=\int_{b_n}^{x/b_n+b_n} V(t) \,
dt-x\le \frac{1}{2}\,\frac{x^2}{b_n^2}+
\log\Big(\frac{x}{b_n^2}+1\Big)< \frac{x^2/2+x}{b_n^2}.
\end{equation}
Moreover, the following bound is immediate:  for $x\ge 0$,
\begin{equation}
\label{desigualtat_y} 0\le e^{-x}(x^2/2+x)\le e^{-\sqrt2}(\sqrt2+1).
\end{equation}
Joining \eqref{nnn3}, (\ref{desigualtat_int}) and
(\ref{desigualtat_y}),   we get
\begin{equation}\label{part0}
\Big\vert\Lambda\left( {I}_n(x)\right)-\Lambda(x)\Big\vert\le
\frac{\sqrt2+1}{e^{\sqrt{2}}b_n^2}.\end{equation} Finally, applying
Proposition~\ref{p5} the result  follows.
\end{proof}

\begin{remark}\label{comp} (i)  Notice that when $x\ge0$, applying
Proposition~\ref{pos} and Proposition~\ref{p4} we have that choosing
$a_n^\circ=\af(b_n)=b_n/(1+b_n^2)$ we get that
\[
\lim_{n_0\to\infty} C^+(n_0)=\frac1{2e}\approx0.184,
\]
while Proposition~\ref{posbis} implies that when
$a_n=\ah(b_n)=1/b_n$ then
\[
\lim_{n_0\to\infty} \overline{C}^+(n_0) =\frac{\sqrt2+1}{2e^{\sqrt
2}}\approx0.294
\]
The above results are coherent with the numerical results presented
in Table~\ref{comparison-t} and show that the first  choice
$a_n^\circ=\af(b_n)$ gives best approximations. We do not develop here
the case $x<0$ for $\ah(b_n)$.

(ii) It is worth  noting  that  since $\ac(t)<1/t$, $\ac(b_n)<\ah(b_n)$, it is
easy to see that Proposition~\ref{posbis} also  holds replacing
$\ah(b_n)$ by $\ac(b_n)$. Nevertheless the bound given by this
result seems less accurate than the ones provided in
Propositions~\ref{pos} and~\ref{posbis}, see
again~Table~\ref{comparison-t}.
\end{remark}

\section{Explicit norming constants}\label{se:5}
Since the expression $b^*_n$ is not  explicit, by using
 asymptotic analysis it can be deduced  expressions asymptotically  equivalent
  for
 $b_n^*$ and $a_n^*$ (they satisfy Property \ref{propietat1})

 \begin{equation}
 \label{beta}
 \beta^*_n=(2\log n)^{1/2}-\frac{\log\log n+\log(4\pi)}{2\,(2\log n)^{1/2}}
 \end{equation}
  and
 $$\alpha^*_n=(2\log n)^{-1/2}.$$
 (see, for example, Resnick \cite[pp. 71--72]{Res87}).
  An easy way to deduce these constants and to  suggest other
 ones more suited to previous results is to use the Lambert W function and its
 extensions.

 \subsection{Lambert $W$ function and extensions}
 \label{subsec:comtet}
 %The real inverse of the function $y\in \R\to y\, e^y$ is called
% the (real) Lambert W function; it has to branches, the principal one
% defined on $(-1/e,\infty)$ and the secondary one defined on
% $(-1/e,0)$.
% We are interested in the principal that we denote by $W$; that
% means

For $t>0$, the equation  $ye^y=t$ has a unique
 real positive solution $y$, which determines (for $t>0$)  the principal branch of the
real Lambert W function, that means,
 $W(t)$ satisfies
$$W(t)\, e^{W(t)}=t,$$
and  $\lim_{t\to\infty}W(t)=\infty$ (see Corless {\it et al.}
\cite{CorGonHarJefKnu86} for a complete overview of Lambert W
function and many applications).
  The asymptotic expansion
of this function is given by Corless {\it et al.} \cite[pp. 22 and
23]{CorGonHarJefKnu86},
 see also
De Bruijn \cite[pp. 25--27]{DeB81}.
\begin{equation}
\label{expansion} W(t)= \log t-\log \log t+\frac{\log\log t}{\log
t}+O\bigg(\Big(\frac{\log\log t}{\log
t}\Big)^2\bigg),\quad  t\to \infty.
%\frac{L_2}{L_1}+\frac{L_2(-2+L_2)}{2L_1^2}+,
\end{equation}
%where
%\begin{equation}
%\label{ls}
%L_1(x)=\log x \quad\text{and}\quad  L_2(x)=\log \vert \log x\vert,
%\end{equation}

For $\gamma\ne 0$, Comtet \cite{Com70} extended that
 expansion to the (unique) positive solution  $y$ of the equation
$$y^\gamma e^y=t$$
such that $y\to\infty$ when $t\to \infty$. Later, Robin \cite{Rob88}
and Salvy \cite{Sal92} extended Comtet \cite{Com70} results in order
to deduce an asymptotic expansion of the solution  of the equation
\begin{equation}
\label{eq:salvi}
y^\gamma e^{y} D\Big(\frac{1}{y}\Big)=t,
\end{equation}
%such that $y\to \infty$ when $t\to \infty$,
 where
$$D(y)=\sum_{n=0}^\infty d_n y^n, \ \text{with} \ d_0\ne 0,$$
is a power series convergent in a neighborhood of the origin. Denote
by $U_D(t)$ that solution. We are only interested on the case
$\gamma=1$ and $d_0=1$, and for this case, Robin \cite{Rob88} and
Salvy \cite{Sal92} prove
\begin{equation*}
 U_D(t)= \log t-\log \log t + \frac{\log \log t-d_1}{\log t} +
\frac{Q_2(\log\log t)}{(\log t)^2}+ o\bigg(\frac{1}{(\log
t)^2}\bigg),\, t\to \infty,
\end{equation*}
where $Q_2$ is a polynomial of degree 2, whose coefficients depend
on $D$. The above expression implies that
\begin{equation}
\label{ugamma} U_D(t)= \log t-\log \log t + \frac{\log \log
t-d_1}{\log t} + O\bigg(\Big(\frac{\log\log t}{\log
t}\Big)^2\bigg),\ t\to \infty.
\end{equation}

%Besides the above references, for more details of that expansions,
% see \cite{GasSalUtz13}.

 \subsection{Return to the norming constants}
 Thanks to (\ref{bnHall}), the constant $b^*_n$ can be written in terms of the principal
 branch of  Lambert function:
 $$b_n^*=\Big(W\big(n^2/(2\pi)\big)\Big)^{1/2}.$$
 Hence, from, (\ref{expansion}),
\begin{align*}
 b_n^*&=\bigg(\log\big(n^2/(2\pi)\big)-\log\log\big(n^2/(2\pi)\big)
 +O\Big(\frac{\log\log n}{\log n}\Big)\bigg )^{1/2}\\
&=(2\log n)^{1/2}-\frac{\log(4\pi\log n)}{2\,(2\log n)^{1/2}}
+O\Big(\frac{(\log \log n)^2}{(\log n)^{3/2}}\Big)\\
&=\beta^*_n+O\Big(\frac{(\log \log n)^2}{(\log n)^{3/2}}\Big),\quad
n\to\infty.
\end{align*}
 Notice that if we
introduce the following sequence
\[
\overline
\beta_n^*=\bigg(\log\big(n^2/(2\pi)\big)-\log\log\big(n^2/(2\pi)\big)
 +\frac{\log\log (n^2/(2\pi))}{\log (n^2/(2\pi))} \bigg )^{1/2},
\]
we obtain a better approximation to $b_n^*$ because
\[
b_n^*=\overline \beta^*_n+ O\left(\dfrac{(\log\log n)^2}{(\log
n)^{5/2}}\right),\quad n\to\infty.\]

In any case,  Theorem \ref{maintheorem} suggests that  the
utilization of an approximation to $b_n$ rather than an
approximation to $b^*_n$ likely  will provide more velocity of
convergence. To this end, in next proposition we compute an
asymptotic expansion of $b_n$ at infinity using the bounds
\eqref{fitesA} for the Mills ratio and the function $U_D$
introduced in~\eqref{ugamma}.

\begin{proposition}\label{asi-bn} It holds that
\begin{equation}\label{appbn}
b_n=\overline \beta_n+O\left(\dfrac{(\log\log n)^2}{(\log
n)^{5/2}}\right),\quad n\to\infty,
\end{equation}
where
\[
\overline
\beta_n=\bigg(\log\big(n^2/(2\pi)\big)-\log\log\big(n^2/(2\pi)\big)
 +\frac{\log\log (n^2/(2\pi))-2}{\log (n^2/(2\pi))} \bigg )^{1/2}.
\]
\end{proposition}

\begin{proof} By inequalities \eqref{fitesA} we know that for $x>0,$
\[
r(x)\phi(x)<1-\Phi(x)<R(x)\phi(x),
\]
where
\[
r(x)=\frac{x}{x^2+1}\quad\mbox{and}\quad R(x)=\frac{x^2+2}{x^3+3x}.
\]
For $n$ large enough, let $v_n$ (resp. $V_n$) be the solution of the
equation $r(x)\phi(x)=1/n$ (resp. $R(x)\phi(x)=1/n$). Recall that
$b_n$ satisfies $1-\Phi(b_n)=1/n.$ Therefore, for these values of
$n$  it holds that
\begin{equation}\label{sand}
v_n\le b_n \le V_n.
\end{equation}
Let us compute the asymptotic expansions at infinity of $\{v_n\}$
and $\{V_n\}.$ Notice that $v_n$ satisfies the equation
\[
\frac{n^2}{2\pi}=x^2e^{x^2}\Big(1+\dfrac1{x^2}\Big)^2=x^2e^{x^2}\Big(1+2\dfrac1{x^2}+\dfrac1{x^4}\Big),
\]
while $V_n$ satisfies
\[
\frac{n^2}{2\pi}=x^2e^{x^2}\left(\frac{1+\frac3{x^2}}{1+\frac2{x^2}}\right)^2=x^2e^{x^2}\left(1+2\dfrac1{x^2}+O\Big(\frac1{x^4}\Big)\right).
\]

By changing $x^2$ by $y$, and  with the notations of Subsection
\ref{subsec:comtet}, both $v_n$ and $V_n$, are
$\big(U_D\big(n^2/(2\pi)\big)\big)^{1/2},$ for some analytic
functions $D$, both satisfying
$$D(y)=1+2{y}+O\big({y^2}\big).$$
Then, from (\ref{ugamma}) and \eqref{sand}, we arrive to the same
asymptotic expansions for $v_n,b_n$ and $V_n.$ Specifically,
$$b_n=\bigg(\log\big(n^2/(2\pi)\big)-\log\log\big(n^2/(2\pi)\big)+
\frac{\log\log\big(n^2/(2\pi)\big)-2}{\log\big(n^2/(2\pi)\big)}+O\bigg(\Big(\frac{\log\log
n}{\log n}\Big)^2\bigg)\bigg)^{1/2}.$$ From the above expression,
\eqref{appbn} follows from
$\sqrt{w+a}-\sqrt{w}=a/(\sqrt{w+a}+\sqrt{w}).$
\end{proof}

\begin{remark}\label{afegit}
 Using the same tools that in the proof of the above
proposition we obtain that
\begin{enumerate} [(i)]
\item $b_n-b_n^*=O\left(\dfrac{1}{(\log
n)^{3/2}}\right),\quad n\to\infty,$
\item $\overline \beta_n-\beta_n^*=O\left(\dfrac{(\log\log n)^2}{(\log
n)^{3/2}}\right),\quad n\to\infty,$
\item $\overline \beta_n-\overline\beta_n^*=O\left(\dfrac{1}{(\log
n)^{3/2}}\right),\quad n\to\infty.$
\end{enumerate}
\end{remark}

\begin{remark} It is also possible to construct some approximations of
$b^*_n$ and $b_n$ that are  improvements of \eqref{beta} adding some
suitable terms. In fact, if we define:
\begin{align*}
\widetilde \beta_n^*&=\sqrt{2\log n}-\frac{\log(4\pi\log
n)}{2\sqrt{2\log n}}-\frac{\big(\log(4\pi\log
n)\big)^2-4\log(4\pi\log n)}{8\sqrt{(2\log n)^{3}}},\\
\widetilde \beta_n&=\sqrt{2\log n}-\frac{\log(4\pi\log
n)}{2\sqrt{2\log n}}-\frac{\big(\log(4\pi\log
n)\big)^2-4\log(4\pi\log n)+8}{8\sqrt{(2\log n)^{3}}},
\end{align*}
it also holds  that
\[
b_n^*=\widetilde \beta_n^*+O\left(\dfrac{(\log\log n)^2}{(\log
n)^{5/2}}\right)\quad \mbox{and}\quad  b_n=\widetilde
\beta_n+O\left(\dfrac{(\log\log n)^2}{(\log n)^{5/2}}\right) \quad
n\to\infty.
\]
Nevertheless the approximations $\overline \beta^*_n$ and $\overline
\beta_n$, respectively, are sharper, specially for small $n.$
\end{remark}

 \subsection{From $\overline \beta_n$ to $\beta_n$}\label{subsec:numer}

As we have seen in the previous subsection, $\overline \beta_n$ is a
very good approximation of $b_n$. Nevertheless, for each $p,q\in\R$,
if we introduce the new constants
\[
B_n(p,q)=\bigg(\log\big(n^2/(2\pi)\big)-\log\log\big(n^2/(2\pi)\big)+
\frac{\log\big(\log
(n^2)+p\big)-2}{\log\big(n^2\big)+q}\bigg)^{1/2},
\]
it also holds that
\[
b_n=B_n(p,q)+O\left(\dfrac{(\log\log n)^2}{(\log
n)^{5/2}}\right),\quad n\to\infty.
\]
In particular, $\overline
\beta_n=B_n\big(-\log(2\pi),-\log(2\pi)\big).$

To obtain some values of $p$ and $q$ that provide better
approximations to $b_n$, at least for  $n$ in the most used range
$[10,10^5],$ we proceed as follows:  For simplicity we fix
$q=-\log(2\pi)$ and consider $p$ as a free parameter to be
determined. For a given $m\in\N$,  we consider the set of $m-9$
equations
\[
b_{k }-B_{k }\big(p,-\log(2\pi)\big)=0,\quad
 k =10,11,\ldots,m.
\]
The actual values $b_{k }$ are obtained numerically. For each $ k $,
let $p_{k }$ be the solution of the corresponding equation, which is
also obtained numerically. Then we define
\[
\widehat p(m)=\frac 1 {m-9}\sum_{ k =10}^m p_{k }.
\]
Notice that $\widehat p(m)$ can be interpreted as the ``best"
solution for the incompatible system formed by the corresponding
$m-9$ equations. We have obtained that $\widehat
p(10^2)\approx0.59$, $\widehat p(10^3)\approx0.47$, $\widehat
p(10^4)\approx0.47$, and $\widehat p(10^5)\approx 0.52$. These
values suggest to consider $p=1/2$ as a candidate to have an
approximation of $b_n$ that is good both for $n\in[10,10^5]$ and for
$n$ large enough. In short we consider
\[
\beta_n=B_n\big(1/2,-\log(2\pi)\big),
\]
that is precisely the expression \eqref{betafinal} given in the
introduction.

In table \ref{taula.beta} there is a numerical comparison between
all the constants involved in this section for different sample
size. These results   illustrate that the suggested new constant
$\beta_n$ is a very good approximation for $b_n$, and that it is
sharper than $\overline \beta_n$, specially for small values of $n$.
Also
 $\overline\beta_n^*$ is a good approximation to $b_n^*$,
whereas $\beta^*_n$  approximates $b_n^*$, but more slowly.
 The computations to get
the table are done with Maple.

\begin{table}[htb]
\centering
\begin{tabular}{rcccccc}
\toprule
{$\bs n$} & $10$  & $10^2$ & $10^5 $ & $10^{10}$   & $10^{30} $ & $10^{60} $ \\
$b_n$ &  1.28155 & 2.32635 & 4.26489 & 6.36134& 11.46402 & 16.39728 \\
$\beta_n$ & 1.27115& 2.32632& 4.26488& 6.36132& 11.46402 & 16.39728\\
$\overline\beta_n$ & 1.18090& 2.31828& 4.26430& 6.36123& 11.46401 & 16.39728\\
\midrule
$b^*_n$&1.43165& 2.37533& 4.27575& 6.36492 & 11.46467& 16.39750 \\
$\overline \beta^*_n$ & 1.45508 & 2.37607 & 4.27535 & 6.36478 &
11.46465 & 16.39750\\
$\beta^*_n$ & 1.36192 & 2.36625 & 4.28019 & 6.36855 & 11.46611 & 16.39821 \\
\bottomrule
\end{tabular}
\caption{Comparison of the standard constants $\beta^*_n$ and the
constants $\overline\beta^*_n$ with $b^*_n$ and the proposed
constants $\overline\beta_n$ and $\beta_n$ with $b_n$.}
\label{taula.beta}
\end{table}

\subsection*{Conclusions}

As a corollary of Theorem \ref{maintheorem}, Proposition
\ref{asi-bn} and the computations of this section, we propose
\eqref{betafinal},
\begin{equation*}
\beta_n=\bigg(\log\big(n^2/(2\pi)\big)-\log\log\big(n^2/(2\pi)\big)+
\frac{\log\big(\log
(n^2)+1/2\big)-2}{\log\big(n^2/(2\pi)\big)}\bigg)^{1/2},
\end{equation*}
that is a very good approximation for $b_n$, as one of the norming
constants for the maximum of $n$ i.i.d. standard normal random
variables. Also, in agreement with Remark \ref{comp} and Table
\ref{comparison-t}, it seems also convenient to utilize always
$A_{\cal F}$. So we propose, instead of $\alpha_n^*$ to use,
together with $b_n$, the norming constant
$$\alpha_n=\frac{\beta_n}{1+\beta_n^2}.$$

\section*{Acknowledgments}
The first author was   partially
 supported by grants MINECO/FEDER reference MTM2008-03437
 and Generalitat de Catalunya reference 2009-SGR410.  The second and third authors by grants
  MINECO/FEDER reference
   MTM2009-08869 and MINECO reference  MTM2012-33937

\end{document}